\documentclass{gtart}
\gtart

\usepackage{amsthm}
\usepackage{amsgen}
\usepackage{amsmath,amssymb}
\usepackage{curvesls}
\usepackage{epic}
\usepackage[all]{xy}
\usepackage[dvips]{graphicx}
\usepackage{psfrag}
\usepackage[T1]{fontenc}
\usepackage[sc]{mathpazo}
\usepackage{color, xcolor, colortbl}
\usepackage{float}
\usepackage{arydshln, multirow} 
\usepackage{multicol}
\usepackage{fancybox}
\newtheorem{thm}{Theorem}[section]

\newtheorem{cor}[thm]{Corollary}
\newtheorem{prop}[thm]{Proposition}

\theoremstyle{definition}

\newcommand\scalemath[2]{\scalebox{#1}{\mbox{\ensuremath{\displaystyle #2}}}}

\newcommand{\Real}{{\mathbb R }}

\newcommand{\Emb}{{\mathrm{Emb}}}
\newcommand{\Diff}{{\mathrm{Diff}}}
\newcommand{\Imm}{{\mathrm{Imm}}}

\newcommand{\HD}{{\mathrm{HD}}}
\newcommand{\holim}{{\mathrm{holim}}}
\newcommand{\EK}[1]{{\mathrm{EC}({#1})}}

\definecolor{orange}{RGB}{255, 123, 21}

\begin{document}

\title{Stabilisation, scanning and handle cancellation}
\authors{Ryan Budney}

\addresses{
Mathematics and Statistics, University of Victoria PO BOX 3060 STN CSC, Victoria BC Canada V8W 3R4}
\emails{rybu@uvic.ca}

\begin{abstract} 
In this note we describe a family of arguments that link the homotopy-type of a) the diffeomorphism group of
the disc $D^n$, b) the space of co-dimension one embedded spheres in $S^n$ and
c) the homotopy-type of the space of co-dimension two trivial knots in $S^n$.  We also describe some natural
extensions to these arguments.  We begin with Cerf's `upgraded' proof of 
Smale's theorem, that the diffeomorphism group of $S^2$ has the homotopy-type of the isometry group.  
This entails a canceling-handle construction, related to the `scanning' maps of spaces of embeddings \cite{BG}
$\Emb(D^{n-1}, S^1\times D^{n-1}) \to \Omega^j \Emb(D^{n-1-j}, S^1 \times D^{n-1})$. 
We further give a Bott-style variation on Cerf's construction, and a related Embedding Calculus framework 
for these constructions.  We use these arguments to prove that the monoid of Sch\"onflies spheres
$\pi_0 \Emb(S^{n-1}, S^n)$ is a group with respect to the connect-sum operation, for all $n \geq 2$. This
last result is perhaps only interesting when $n=4$, as when $n \neq 4$ it follows from the resolution of the various
generalized Sch\"onflies problems. 
\end{abstract}

\maketitle


\section{Introduction}\label{intro}

In Cerf's landmark paper \cite{Ce2}, somewhat overlooked is a novel proof of Smale's theorem, that the group of diffeomorphisms
of the $2$-sphere, $\Diff(S^2)$ has the homotopy-type of its linear subgroup $O_3$.  The core of Cerf's argument is the proof that
the Smale-Hirsch map (pointwise derivative) $\Diff(D^2) \to \Omega^2 GL_2(\Real)$ has a left homotopy-inverse.  Cerf states his theorem 
in the language of homotopy groups, i.e. the homotopy groups 
of $\Diff(D^2)$ inject into the homotopy groups of $\Omega^2 GL_2(\Real)$.  Since the latter homotopy groups are trivial, and 
diffeomorphism groups of compact manifolds have the homotopy-type of countable CW-complexes \cite{Pa2}, this allows Cerf to conclude $\Diff(D^2)$
is contractible via the Whitehead Theorem. In this paper we use the notation that if $N$ is a manifold with boundary, $\Diff(N)$ denotes the group of 
diffeomorphisms of $N$ that restrict to the identity on $\partial N$.   We will use the same conventions for embedding spaces, i.e.
$\Emb(N, M)$ denotes the space of smooth embeddings of $N$ in $M$, and if $N$ and $M$ have boundary these maps will all restrict to 
one given map $\partial N \to \partial M$. 

Smale's proof that $\Diff(D^2)$ is contractible uses the Poincar\'e-Bendixson Theorem to guarantee the flow of the 
vector fields he uses terminate in finite time.  As the Poincar\'e-Bendixson
theorem is only available in dimension two, it limits the applicability of Smale's argument.   We should note, there
have been attempts to broaden the applicability of a Smale-type argument, by studying spaces of closed $1$-forms. 
See for example the two papers of Laudenbach and Blank \cite{La1} \cite{La2} for a sampling.  Since Cerf's argument does not depend on
Poincar\'e-Bendixson, it allows for greater applicability.  The headline consequences of Cerf's arguments
are that the diffeomorphism group $\Diff(D^n)$ has the same homotopy-type  as $\Omega \Emb(D^{n-1}, D^n)$ and also that 
the embedding space $\Emb(D^{n-1}, D^n)$ is a homotopy-retract of $\Omega \Emb(D^{n-2}, D^n)$.  Putting these two results
together, the homotopy groups of $\Diff(D^n)$ inject into the homotopy groups of $\Omega^2 \Emb(D^{n-2}, D^n)$ for all $n$.  
While Cerf states these as his theorems, his techniques prove much more.  It is the purpose of this paper to 
outline the consequences of his techniques. 

\subsection{Cerf's techniques}

To give Cerf's results some context, we first mention how the spaces he studies are related to some more commonly-discussed
objects.  A linearization argument \cite{fam} shows that the diffeomorphism group $\Diff(S^n)$ has the homotopy-type
of $O_{n+1} \times \Diff(D^n)$, indeed the homotopy-equivalence comes from considering $\Diff(D^n)$ as the subgroup
of $\Diff(S^n)$ that is the identity on a fixed hemi-sphere, and the homotopy-equivalence
$O_{n+1} \times \Diff(D^n) \to \Diff(S^n)$ is given by the group multiplication in $\Diff(S^n)$. 
There is an analogous homotopy-equivalence $\Emb(S^j, S^n) \simeq SO_{n+1} \times_{SO_{n-j}} \Emb(D^j, D^n)$
when $n>j$.  If we let $\Emb(D^{n-1}, S^1 \times D^{n-1})$ denote the space of smooth embeddings of 
$D^{n-1}$ in $S^1 \times D^{n-1}$ which agree with the standard inclusion $D^{n-1} \to \{1\} \times D^{n-1}$
on the boundary sphere, then there is a `handle-filling' homotopy-equivalence
$\Diff(S^1 \times D^{n-1}) \simeq \Diff(D^n) \times \Emb(D^{n-1}, S^1 \times D^{n-1})$. 

In Cerf's paper \cite{Ce2} the main results we highlight concern three maps:
\begin{enumerate}
\item $\Diff(D^n) \to \Omega \Emb(D^{n-1}, D^n)$.
\item $\Emb(D^{n-1}, D^n) \to \Emb(D^{n-1}, S^1 \times D^{n-1})$, this is the map given by attaching a $1$-handle to $D^n$
so that the attaching sphere links the standard $S^{n-2}$ in $\partial D^n \equiv S^{n-1}$, i.e. we think of $S^1 \times D^{n-1}$
as $D^n$ with a $1$-handle attached, thus the map comes from simply changing the co-domain of the embedding.
\item $\Emb(D^{n-1}, S^1 \times D^{n-1}) \to \Omega \Emb^\nu(D^{n-2}, D^n)$. The symbol $\nu$ indicates the embeddings
are required to have an everywhere non-zero normal vector field, and the vector field are some standard (constant)
on the boundary.
\end{enumerate}

Cerf's result is that the maps (1) and (3) are homotopy-equivalences, while (2) is a homotopy-retract, i.e. has 
a left homotopy-inverse.  The definition of the maps (1) and (3) are analogous, and will be described 
precisely in Section \ref{mainsec}. The rough idea of these maps is to fiber the domain of the embedding by a 
$1$-parameter family of co-dimension one discs, and restrict the map to these fibers, and appropriately
changing the co-domain of the family of embeddings, via a filling.  In the case of (3) 
the fibering construction would give a $1$-parameter family of embeddings of $D^{n-2}$ in $S^1 \times D^{n-1}$ but we carefully 
fill with a canceling $2$-handle, to construct an element of $\Omega \Emb(D^{n-2}, D^n)$.  

\subsection{Extrapolating from Cerf}

In Section \ref{mainsec} we observe that Cerf's argument, unchanged, gives a homotopy-equivalence
$$\Emb(D^j, S^{n-j} \times D^j) \to \Omega \Emb^\nu(D^{j-1}, D^n).$$
Cerf's results (1) and (3) above correspond to the $j=n$ and $j=n-1$ case of this homotopy-equivalence. 
If we think of $S^{n-j} \times D^j$ as $D^n$ union an $(n-j)$-handle, then the domain of our map, 
$\Emb(D^j, S^{n-j} \times D^j)$ is the
space of all cocores, i.e. smooth embeddings of $D^j$ in $S^{n-j} \times D^j$ that agree with a standard
linear inclusion $D^j \to \{*\} \times D^j$ on the boundary. 
The codomain is the loop space of the space of smooth embeddings $D^{j-1} \to D^n$ that carry a nowhere-zero
normal vector field, moreover the embedding and the vector field are standard linear embeddings on the boundary. 
The base-point of the embedding space $\Emb^\nu(D^{j-1}, D^n)$ is the linear (i.e. boundary parallel) embedding.

 It is here where authors noticed a connection to
recent `lightbulb theorems' in low-dimensional topology \cite{Ga} \cite{BG} \cite{KT}. The above equivalence can be recast
slightly -- using the same argument but applying it to a strictly larger class of spaces.  
Let $N$ be a compact $n$-manifold with non-empty boundary, and let $\natural$ denote the boundary 
connect-sum operation. Think of the boundary connect sum $N \natural (S^{n-j} \times D^j)$ as $N$ with a
trivial $(n-j)$-handle attached, then the space of cocores of this attached handle, which we could denote
$\Emb(D^j, N \natural (S^{n-j} \times D^j))$ has the same homotopy-type as the loop space 
$\Omega \Emb^\nu(D^{j-1}, N)$, which is the loop space of the space of embedded $D^{j-1}$ discs
with normal vector field in the manifold $N$ -- the space of embeddings we give the base-point of
a boundary-parallel embedding.   This version is emphasized in \cite{KT}.

\subsection{Related expositions}

Another way to look at this paper is that it is both an addendum to \cite{fam}, and a paper that highlights
methods from \cite{BG} and \cite{Ce2} that deserve to be singled-out.  Both \cite{BG} and \cite{Ce2} are long papers with many
results, so it is easy to overlook this technique. We hope a shorter-format paper devoted to one tool does the ideas the justice 
they deserve.  In \cite{fam}, an attempt was made to describe the most basic relations between the homotopy-type of diffeomorphism 
groups and embedding spaces for the smallest manifolds, such as spheres and discs.  These Cerf techniques were known to the 
author, but perhaps indicative of the techniques, the only consequences the author knew at the time were 
already known, by other methods.  So they were removed from the paper before publication.   

For example, the connection between the 
homotopy-type of the component of the unknot $\Emb_u(S^1, S^3)$, with the homotopy-type of 
$\Diff(S^3)$, which is immediate from Cerf's perspective, is historically derived using Hatcher's work on spaces of 
incompressible surfaces \cite{Hat} (see the final pages).  In Section \ref{miscsec} we describe the relation between
Cerf's half-disc fibrations and the more commonly used restriction fibration $\Diff(S^n) \to \Emb(S^j, S^n)$. 

\subsection{Sch\"onflies}

An interesting observation in \cite{BG} is that the `stacking' operation, while appearing to 
be just a monoid structure on the space $\Emb(D^{n-1}, S^1 \times D^{n-1})$, using Cerf's argument one can show the space is 
group-like, i.e. the induced monoid structure on $\pi_0 \Emb(D^{n-1}, S^1 \times D^{n-1})$ is that of a group, for all $n \geq 2$.  
One consequence of this is an argument the monoid of Sch\"onflies spheres $\pi_0 \Emb(S^{n-1}, S^n)$ is a group using the 
relative connect-sum operation.  There is a classical argument due to Kervaire-Milnor that this monoid has inverses. Our argument
is characteristically different, in that we construct an onto homomorphism from a group, i.e. in a weak sense we give a presentation of
the monoid of Sch\"onflies spheres.  This appears in Section \ref{schoen}. 

\subsection{High co-dimension scanning}

A scanning technique was proposed for studying the homotopy-type of $\Diff(S^1 \times D^n)$, 
by considering the chain of maps
$$\scalemath{0.9}{\Diff(S^1 \times D^{n-1}) \to \Emb(D^{n-1}, S^1 \times D^{n-1}) \to 
\Omega \Emb(D^{n-2}, S^1 \times D^{n-1}) \to \cdots \to \Omega^{n-2} \Emb(D^1, S^1 \times D^{n-1})}$$
in the sequence \cite{BG}, \cite{BG2}. Interestingly, an infinitely-generated subgroup of 
$\pi_{n-4} \Diff(S^1 \times D^{n-1})$ survives to the end
of the sequence
$$\pi_{n-4} \Omega^{n-2} \Emb(D^1, S^1 \times D^{n-1}) \equiv \pi_{2n-6} \Emb(D^1, S^1 \times D^{n-1}),$$
 for all 
$n \geq 4$.  At present little is known about Cerf's scanning maps $\Diff(D^n) \to \Omega^j \Emb(D^{n-j}, D^n)$ when 
$j \geq 3$, but these results suggest such maps have the potential to be homotopically non-trivial, and could be
used to deduce results even about $\pi_0 \Diff(D^n)$ for $n \geq 4$.  Although, we now know the
map $\Diff(D^n) \to \Omega^{n-1} \Emb(D^1, D^n)$ is null-homotopic \cite{BG2}.  The transitional map 
$$\Emb(D^{n-2}, D^n) \to \Omega \Emb(D^{n-3}, D^n)$$
is perhaps of greatest interest, as the target space can be studied with techniques such as the Embedding Calculus, while
we have little in the way of general theory to study the homotopy-type of $\Emb(D^{n-2}, D^n)$. It would be more precise to
to say we have general theory when $n < 4$ but when $n \geq 4$, separating the path-components of $\Emb(D^{n-2}, D^n)$ is
a difficult mathematical problem. 
Similarly, little is known about $\pi_1 \Emb(D^2, D^4)$ at present.   If one allows the embeddings to have trivialized 
normal bundles (normal framings) one has scanning maps of the form 
$\Diff(D^n) \to \Omega^j \Emb^{fr}(D^{n-j}, D^n) \to \Omega^n GL_n(\Real)$ where the
space on the right is the terminal $j=n$ case. The map $\Diff(D^n) \to \Omega^n GL_n(\Real)$ is known as the Smale-Hirsch
map, i.e. the pointwise derivative map.  Whether or not this Smale-Hirsch map is homotopically non-trivial has been an 
open problem for some time.  Interestingly, it has recently been shown to be homotopically non-trivial in 
the $n=11$ case \cite{Cr}.

One other impetus for studying such scanning maps is that these embedding spaces are highly structured objects. For example, 
$\Diff(D^n)$ is homotopy-equivalent to the space $\EK{n, *}$, called the `cubically-supported embedding space'. 
If $M$ is a compact manifold, 
$\EK{j,M}$ denotes the space of smooth embeddings $f : \Real^j \times M \to \Real^j \times M$ where the support $supp(f)$
is constrained to be a subset of $I^j \times M$, i.e. $supp(f) = \{ p \in \Real^j \times M : f(p) \neq p \} \subset I^j \times M$. 
The space $\EK{j, M}$ admits an action of the operad of $(j+1)$-cubes, thus it is not far away from being an $(j+1)$-fold loop-space.
The way to think about this operad action is there is an action of the $j$-cubes operad on $\EK{j, M}$, due to the 
affine structure on the 
$\Real^j$ factors of $\Real^j \times M$.  The space $\EK{j,M}$ is also a monoid under composition of functions, 
and these two operations can be
promoted naturally to a $(j+1)$-cubes action, described in \cite{fam}. 

The space $\EK{j, D^{n-j}}$ fibers over $\Emb(D^j, D^n)$ with fiber $\Omega^{j} SO_{n-j}$ -- indeed, the spaces
$\EK{j, D^{n-j}}$ are homotopy-equivalent to $\Emb^{fr}(D^j, D^n)$.  There are scanning maps 

$$\EK{n, *} \to \Omega \EK{n-1, D^1} \to \cdots \to \Omega^j \EK{n-j, D^j} \to \cdots \to \Omega^{n-1} \EK{1, D^{n-1}} \to \Omega^n GL_n(\Real)$$
which commute with the action of the $(n+1)$-cubes operad.  While the reference \cite{fam} allows one to see these cubes actions
explicitly, there are also ways of describing the iterated loop space structure using smoothing theory. 
Thus the ability of scanning maps to detect homotopy in diffeomorphism groups and embedding spaces is closely connected to the 
question of to what extent the Smale-Hirsch map for $\Diff(D^n)$ is non-trivial.  
To add some additional context, iterated loop spaces are highly structured objects, and 
finding maps between them is somewhat analogous to finding a homomorphism between other highly-structured objects like rings or modules: 
if the map is not zero, it is often highly non-trivial.

In this paper we outline what is known about such scanning maps, and where some potentially interesting future 
computations sit. 

The author would like to thank David Gabai, Robin Koytcheff, Victor Turchin, Hyam Rubinstein and Alexander Kupers for helpful comments.
In particular, this paper is largely exposition of a subset of results from a joint paper with David Gabai \cite{BG}. 


\section{Canceling handles}\label{mainsec}

The space $\Emb(D^j, N)$ denotes the space of embeddings of $D^j$ in $N$ where the boundary of $D^j$ is mapped to 
$\partial N$ in some fixed, prescribed manner.  In the case of $\Emb(D^j, D^n)$ the embedding is required to restrict 
to the standard inclusion $x \longmapsto (x,0)$ on the boundary. 

Cerf constructs an isomorphism \cite{Ce2} (Proposition 5, pg. 128)  for all $i \geq 0$ and $n \geq 1$ (see also
Theorem 4 of \cite{Ce1})

$$\pi_i \Diff(D^n) \simeq \pi_{i+1} \Emb(D^{n-1}, D^n)$$
 which we promote to a homotopy-equivalence $\Diff(D^n) \simeq \Omega \Emb(D^{n-1}, D^n)$.  Equivalently,
this homotopy-equivalence can be stated as a description of the classifying space of $\Diff(D^n)$
$$B\Diff(D^n) \simeq \Emb_u(D^{n-1},D^n).$$
The subscript $u$ indicates the component of the unknot in $\Emb(D^{n-1}, D^n)$, i.e. the component of the linear embedding.  The above
results were stated at least as far back as \cite{fam}, but it would not be surprising if this observation 
had been written
down earlier.  

The map $\Diff(D^n) \to \Omega \Emb(D^{n-1}, D^n)$ 
has a simple description thinking of $\Diff(D^n)$ as the diffeomorphisms of $\Real^n$ with support
contained in $D^n$.  One then considers $D^n$ as a subset of $I \times D^{n-1}$.  Restriction to the
fibers $\{t\} \times D^{n-1}$ gives the $1$-parameter family of embeddings of $D^{n-1}$ into $D^n$. 
After suitably translating and scaling the embedding family to have fixed boundary conditions, this is an 
element of $\Omega \Emb(D^{n-1}, D^n)$. 

The map back
$\Omega \Emb(D^{n-1}, D^n) \to \Diff(D^n)$ is defined by an elementary isotopy-extension construction. 
Following Cerf, let $\HD^j$ denote the $j$-dimensional {\it half-disc}, i.e. 
$$\HD^j = \{ (x_1, \cdots, x_j) \in \Real^j : \sum_{i=1}^j x_i^2 \leq 1, x_1 \leq 0 \}.$$
The boundary $\partial \HD^j$ consists of the two subspaces: The subspace (1) $\partial D^j \cap \HD^j$,
called the {\it round face}, and the subspace (2) satisfying $x_1=0$ called the {\it flat face}.

\begin{figure}[H]
{
\psfrag{HD}[tl][tl][0.7][0]{$f(\HD^n)$}
\psfrag{FF}[tl][tl][0.7][0]{$f(\{0\} \times D^{n-1})$}
\psfrag{Dn}[tl][tl][0.7][0]{$D^n$}
$$\includegraphics[width=10cm]{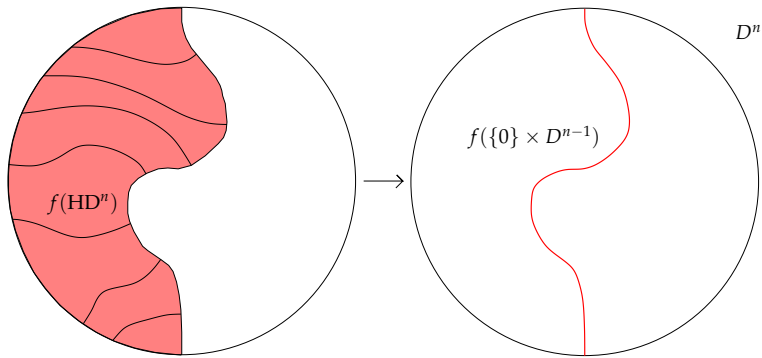}$$
}
\caption{\label{HDfib} The half-disc fibration.}
\end{figure}

Let $\Emb(\HD^n, D^n)$ be the space
of embeddings of $\HD^n$ into $D^n$ that restrict to the identity map on $\HD^n \cap \partial D^n$, i.e.
acting as the identity on the round face. The map given by restriction to the flat face
is a Serre fibration (see Figure \ref{HDfib}) \cite{Ce0} 
$$\Diff(\HD^n) \to \Emb(\HD^n, D^n) \to \Emb(D^{n-1}, D^n).$$
Moreover, via an argument directly analogous to the homotopy classification of collar neighborhoods or
tubular neighborhoods, one can show  $\Emb(\HD^n, D^n)$ is contractible \cite{Ce1}.  The rough idea is that every 
such embedding is isotopic to its restriction to a small neighborhood of the round face, where you can approximate
the embedding by the standard linear inclusion -- indeed, the straight-line homotopy between the embedding
and the standard linear inclusion is an isotopy, at least in a sufficiently-small neighborhood of the round face.

The proof that the above map is a Serre fibration is a version of the isotopy-extension theorem `with parameters', i.e. 
the proof of isotopy extension given in Hirsch's text \cite{Hir} suffices to also prove such maps are Serre fibrations. 
We should also mention that Palais also has shown \cite{Pa2} that a broad class of spaces of embeddings and 
diffeomorphism groups, including all the spaces discussed in this paper, have the homotopy-type of countable
CW-complexes.  The rough idea of the proof is that such embedding spaces are homeomorphic to open subsets of 
a Hilbert cube (consider for example representing smooth functions via something like a Fourier expansion), and open subsets of
Hilbert cubes admit CW-structures, in a manner analogous to open subsets of $\Real^n$.  

The total space $\Emb(\HD^n, D^n)$ is contractible, as sketched above and proven by Cerf \cite{Ce2}.  This tells us that 
the connecting map 
$$\Omega \Emb(D^{n-1}, D^n) \to \Diff(\HD^n)$$
is a homotopy-equivalence.  The inclusion $\Diff(\HD^n) \to \Diff(D^n)$ is a homotopy-equivalence via a rounding-the-corners
argument.  The definition of the connecting map $\Omega \Emb(D^{n-1}, D^n) \to \Diff(\HD^n)$ comes from
observing that an element of $\Omega \Emb(D^{n-1}, D^n)$ via currying can be thought of as a map
$[0,1] \times D^{n-1} \to D^n$ which is continuous globally, but smooth on the $\{t\} \times D^{n-1}$ fibers. 
A smoothing construction \cite{Hir} allows us to perturb this map to be globally smooth, not affecting the
the restriction of the map to the boundary of $[0,1] \times D^{n-1}$. It is with this smoothing that we apply
the isotopy-extension construction. Specifically, this smoothing argument tells us the subspace of 
$\Omega \Emb(D^{n-1}, D^n)$ such that the associated map $[0,1] \times D^{n-1} \to D^n$ is smooth, this subspace 
has the same homotopy-type as $\Omega \Emb(D^{n-1}, D^n)$.  There is an alternative approach that is formally 
analogous to the result that the loop space of a manifold has the same homotopy-type as its subspace of smooth loops. 
We view $\Emb(D^{n-1}, D^n)$ as a smooth Banach or Fr\'echet manifold (depending on the 
order of differentiability of the embeddings, $C^k$ with $k$ finite or infinite respectively).  
From this perspective a smooth map $[0,1] \to \Emb(D^{n-1}, D^n)$ via currying produces a smooth
map $[0,1] \times D^{n-1} \to D^n$. This has been made precise in several places in
the literature, see \cite{Kh} or \cite{Mi}.

We can justify why scanning $\Diff(D^n) \to \Omega \Emb(D^{n-1}, D^n)$ is the homotopy-inverse
to the connecting map $\Omega \Emb(D^{n-1}, D^n) \to \Diff(D^n)$ via Figure \ref{scanfig}.  We have
a central square whose horizontal axis is labeled $t$, and whose vertical axis is labeled $\alpha$. 
Given $t \in [0,1]$ let $f_t : \HD^n \to D^n$ denote the embedding whose restriction to $\{0\} \times D^{n-1}$
is a given element of $\Omega \Emb(D^{n-1}, D^n)$.  Let $\simeq$ be the equivalence relation on $[0,1] \times D^{n-1}$
generated by the equivalence classes $[0,1] \times \{p\}$ for all $p \in \partial D^{n-1}$, thus
$[0,1] \times D^{n-1}$ can be identified with $\HD^n$, i.e. we collapse all the edges $[0,1] \times \{p\}$ for all
$p \in \partial D^{n-1}$.  Under the identification of $[0,1] \times D^{n-1} / \sim$ with $\HD^n$, the corner strata
corresponds to the collapsed edges, $\{0\} \times D^{n-1}$ to the round face, and $\{1\} \times D^{n-1}$ to the flat
face.  In the Figure \ref{scanfig} $f_t(\{ \alpha \} \times D^{n-1})$ is denoted via a thick red curve.  The upper
line of our square therefore denotes $f_t(\{0\} \times D^{n-1})$, our element of $\Omega \Emb(D^{n-1}, D^n)$.
This is homotopic to the concatenation of the other three boundary segments of the square.  The rightmost segment of 
the square is the `sweep-out' portion of scanning, and the left-most segment is the sweep-out of the standard inclusion. 
The lower edge is constant. 

\begin{figure}[H]
{
\psfrag{alpha}[tl][tl][0.7][0]{$\alpha$}
\psfrag{time}[tl][tl][0.7][0]{$t$}
$$\includegraphics[width=12cm]{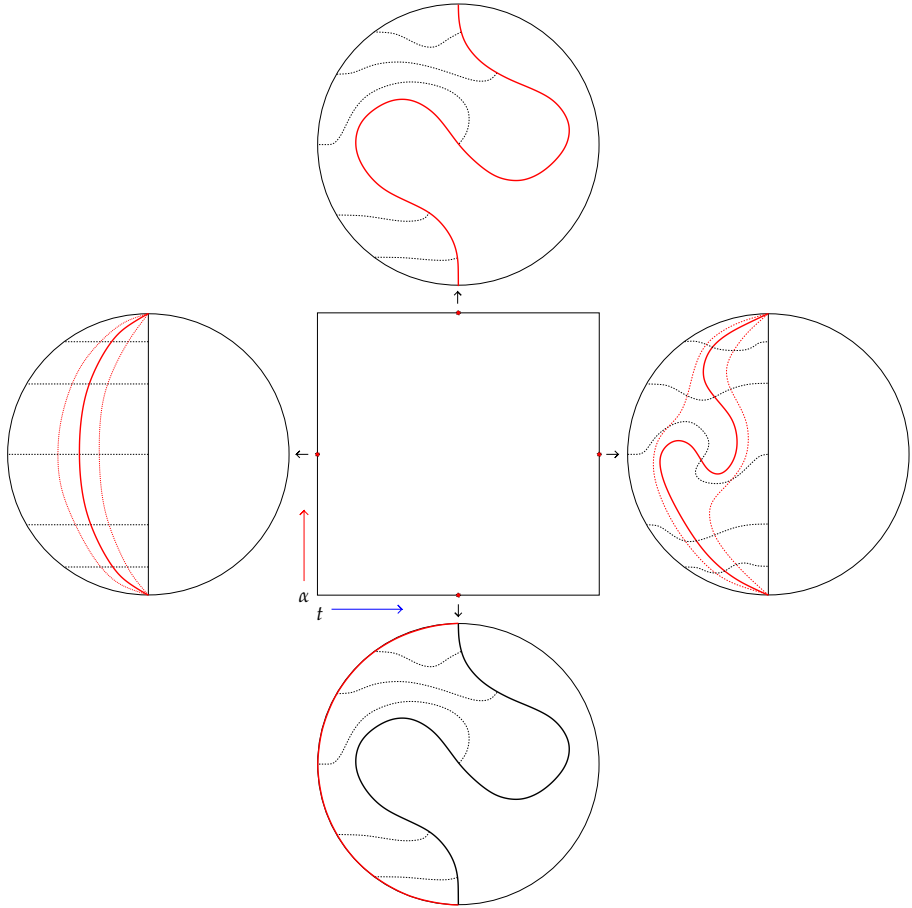}$$
}
\caption{\label{scanfig} Homotopy-inverse of isotopy extension.}
\end{figure}

To extrapolate,  let $\Emb^{\nu}(D^{j-1}, N)$ denote the space of smooth embeddings of $D^{j-1}$ in $N$ such that the boundary 
is sent to the boundary in a prescribed manner, and the embedding comes equipped with a normal vector field (standard on the boundary), then we have 
a restriction (Serre) fibration
$$\Emb(D^j, N \setminus \nu D^{j-1}) \to \Emb(\HD^j, N) \to \Emb^{\nu}(D^{j-1}, N).$$
The total space is the space of smooth embeddings of $\HD^j$ in $N$ such that the round face is sent
to $\partial N$ in a prescribed manner. The space $\nu D^{j-1}$ indicates an open tubular neighborhood in $N$ corresponding
to the base-point element of $\Emb^{\nu}(D^{j-1}, N)$.  We keep track of the normal vector field in the base space, 
as otherwise the fiber would be an embedding space where the discs are not neatly embedded.  
One can of course argue the above is not literally the fiber -- it should be the subspace
of $\Emb(\HD^j, N)$ that agrees with a fixed embedding on the flat boundary. That said, blowing up the flat boundary or
a tubular neighborhood argument together with drilling the open tubular neighborhood completes the identification of
the fibre.

This gives us an analogous homotopy-equivalence
$$\Omega \Emb^{\nu}(D^{j-1}, N) \simeq \Emb(D^j, N \setminus \nu D^{j-1}).$$

\begin{figure}[H]
{
\psfrag{D}[tl][tl][0.7][0]{$\textcolor{blue}{S^{n-j} \times \{*\} }$}
\psfrag{pxD}[tl][tl][0.7][0]{$\textcolor{red}{\{*\} \times \partial D^j}$}
\psfrag{SxD}[tl][tl][1][0]{$S^{n-j} \times D^j$}
\psfrag{fD}[tl][tl][1][0]{$\textcolor{red}{f(D^j)}$}
\psfrag{N}[tl][tl][1][0]{$N$}
$$\includegraphics[width=12cm]{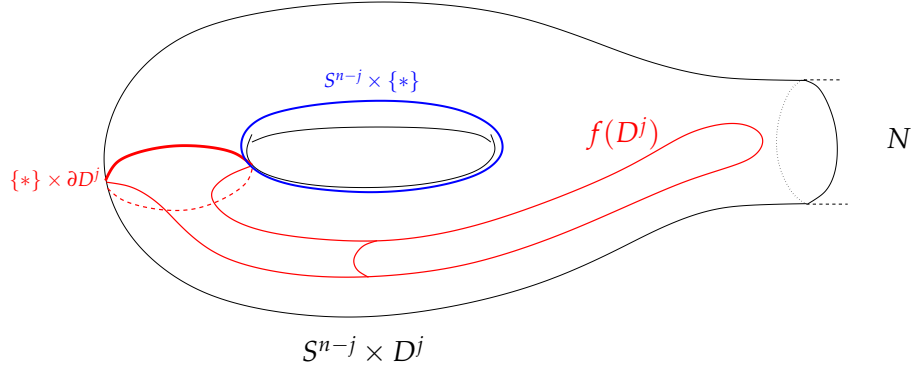}$$
}
\caption{\label{chand} Cocore embedding $f \in \Emb(D^j, (S^{n-j} \times D^j) \natural N)$ in red.
If one drills a tubular neighborhood of a linearly embedded $D^{j-1} \to D^n$ one has a manifold diffeomorphic
to $S^{n-j} \times D^j$, which gives the equivalence
$(S^{n-j} \times D^j) \natural N \simeq N \setminus \nu D^{j-1}$.}
\end{figure}

The space $N \setminus \nu D^{j-1}$ is $N$ with a $(j-1)$-handle drilled-out, and the embedding of $D^j$
is a canceling handle for the $(j-1)$-handle, thus the $(j-1)$-handle is parallel to the boundary.  
As another model for $N \setminus \nu D^{j-1}$, we turn the handle upside-down and think of this manifold
as $N \cup H^{n-j}$, i.e. $N$ union a $(n-j)$-handle.  Since the handle attachment is trivial, this manifold is diffeomorphic to 
$(S^{n-j} \times D^j) \natural N$.  From this perspective, the embeddings of
$\Emb(D^j, (S^{n-j} \times D^j) \natural N)$
can be thought of as a space of embeddings of cocores for the $(n-j)$-handle attachment of the boundary 
connect-sum $(S^{n-j} \times D^j) \natural N$, i.e. these cocores are allowed to reach into the $N$ summand.

 This last interpretation is perhaps the most convenient for stating the homotopy-equivalence 
$\Omega \Emb^{\nu}(D^{j-1}, N) \simeq \Emb(D^j, (S^{n-j} \times D^j) \natural N)$ as the boundary condition on the latter
embedding space sends $\partial D^j$ to $\{p\} \times \partial D^j \subset S^{n-j} \times D^j$.  By design, the embeddings in 
$\Emb^{\nu}(D^{j-1}, N)$ are isotopically trivial on the boundary $S^{j-2} \to \partial N$. 

\begin{thm}\label{main1}
There is a homotopy-equivalence
$$\Omega \Emb^{\nu}(D^{j-1}, N) \simeq \Emb(D^j, (S^{n-j} \times D^j) \natural N)$$
where $\Emb^{\nu}(D^{j-1}, N)$ is the space of smooth embeddings of $D^{j-1}$ in $N$ such that the pre-image of the boundary
of $N$ is the boundary of $D^{j-1}$. The embedding of $\partial D^{j-1}$ is required to be a fixed embedding, and isotopically
trivial, i.e. bounding an embedded $D^{j-1} \to \partial N$. The base-point of $\Emb^\nu(D^{j-1}, N)$ can be chosen to be
any embedding that is parallel to an embedding in $\partial N$ (rel $\partial$).  
The $\nu$ indicates the embedding comes equipped with a normal vector
field, standard on the boundary.  The space $\Emb(D^j, (S^{n-j} \times D^j) \natural N)$ is a space of 
cocores for the handle attachment 
$(S^{n-j} \times D^j) \natural N = N \cup H^{n-j}$, i.e. it is the space of smooth embeddings of 
$D^j$ in $(S^{n-j} \times D^j) \natural N$
such that the boundary of $D^j$ is sent to $\{*\} \times \partial D^j$ where $* \in S^{n-j}$ is some point disjoint from the
mid-ball of the boundary connect-sum. 
\end{thm}

Alternatively, one could express the theorem in the `reductionist' form
$$\Emb(D^j, M) \simeq \Omega \Emb^{\nu}(D^{j-1}, M \cup H^{n-j+1})$$
i.e. by writing $M = (S^{n-j} \times D^j) \natural N$, then 
$N = M \cup H^{n-j+1}$, i.e. we derive $N$ from $M$ by adding a canceling handle.
Thus for the homotopy-equivalence to hold we need $M$ to admit a canceling handle, 
i.e. for an element $f \in \Emb(D^j, M)$ the restriction to the boundary is an embedding
$f_{|\partial D^j} : S^{j-1} \to \partial M$ and there must admit an embedding
$S^{n-j} \to \partial M$ with a trivial normal bundle that transversely intersects $f_{|\partial D^j}$
in a single point.  In the recent `light-bulb theorem' literature the embedded $S^{n-j}$ is simply
called a {\it transverse sphere} \cite{Ga}.  This version of Theorem \ref{main1} appears in \cite{KT}. 

A homotopy-equivalence can be expressed as a map in either direction.  The map
$\Omega \Emb^{\nu}(D^{j-1}, N) \to \Emb(D^j, N \setminus \nu D^{j-1})$ is induced by isotopy extension i.e. one 
lifts the element of $\Omega \Emb^{\nu}(D^{j-1}, N)$ to a path in $\Emb(\HD^j, N)$, starting at the base-point of
$\Emb(\HD^j, N)$. Drilling the flat face from the endpoint of this path gives the element of
$\Emb(D^j, N \setminus \nu D^{j-1})$. 
 
The map back
$\Emb(D^j, N \setminus \nu D^{j-1}) \to \Omega \Emb^{\nu}(D^{j-1}, N)$ involves thinking of $D^j$ as fibered by parallel 
copies of $D^{j-1}$ and taking those restrictions, and composing with the inclusion $N \setminus \nu D^{j-1} \to N$. 
The paper \cite{BG} gives a detailed account
in the $\Emb(\HD^j, D^n)$ case, and \cite{KT} gives a detailed account using the `reductionist' perspective for $\Emb(\HD^j, N)$. 

\begin{prop}\label{cerf1}
The co-dimension $2$ scanning map
$$\Diff(D^n) \to \Omega^2 \Emb^\nu(D^{n-2}, D^n)$$
induces a split injection on all homotopy and homology groups, for $n \geq 2$. The
map admits a left homotopy-inverse.
\end{prop}

Proposition \ref{cerf1} is a space-level statement of Proposition 6 of \cite{Ce2}. 
When  $n=2$, the two-fold scanning map $\Diff(D^2) \to \Omega^2 \Emb^{\nu}(D^0, D^2) \equiv \Omega^2 GL_2(\Real)$
is the Smale-Hirsch map.  Since $\Omega^2 GL_2(\Real)$ is contractible, this is Smale's theorem that
$\Diff(D^2)$ is contractible.  Since $\Diff(S^2) \simeq O_3 \times \Diff(D^2)$ (this is a standard linearisation
argument, see \cite{fam}), this proves Smale's Theorem $\Diff(S^2) \simeq O_3$. 

When $n > 3$ the
forgetful map $\Emb^\nu(D^{n-2}, D^n) \to \Emb(D^{n-2}, D^n)$ is a homotopy-equivalence, since the
fiber has the homotopy-type of $\Omega^{n-2} S^1$. When
$n=2$ or $n=3$, the double-looping of the map $\Omega^2 \Emb^\nu(D^{n-2}, D^n) \to \Omega^2 \Emb(D^{n-2}, D^n)$
is a homotopy-equivalence, as the fiber has the homotopy-type of $\Omega^2 \Omega^{n-2} S^1$. 

\begin{cor}(Smale)\label{smthm} $\Diff(D^2)$ is contractible, i.e. $\Diff(S^2)$ has the homotopy-type of its
linear subgroup $O_3$. 
\end{cor}

\begin{proof}(of Proposition \ref{cerf1}) 
The proof follows from forming
a composite of functions involving
the homotopy-equivalence $\Diff(D^n) \to \Omega \Emb(D^{n-1}, D^n)$ (i.e. Theorem \ref{main1}, $N=D^n$, $j=n$) 
with the induced map on loop spaces from Theorem \ref{main1}, where $N=D^n$ and $j=n-1$, 
$$\Emb(D^{n-1}, S^1 \times D^{n-1}) \to \Omega \Emb^{\nu}(D^{n-2}, D^n).$$
Given that the unit normal fibers are copies of $S^1$, we can discard the normal vector fields, i.e. 
the forgetful map $\Omega^2 \Emb^{\nu}(D^{n-2}, D^n) \to \Omega^2 \Emb(D^{n-2}, D^n)$ is a homotopy-equivalence. 
Think of $S^1 \times D^{n-1}$ as $D^n$ union a $1$-handle, this gives an inclusion 
$\Emb(D^{n-1}, D^n) \to \Emb(D^{n-1}, S^1 \times D^{n-1})$.
Thus we have a composable triple
$$\Diff(D^n) \to \Omega \Emb(D^{n-1}, D^n) \to \Omega \Emb(D^{n-1}, S^1 \times D^{n-1}) \to \Omega^2 \Emb(D^{n-2}, D^n).$$
The left homotopy-inverse of the map in the middle comes from thinking of the universal cover of $S^1 \times D^{n-1}$ as a copy of 
$\Real \times D^{n-1}$ which could also be thought of as as $D^n$ remove two points on its boundary, i.e. we have a map back 
$\Emb(D^{n-1}, S^1 \times D^{n-1}) \to \Emb(D^{n-1}, D^n)$. Since the two maps on the ends are
homotopy-equivalences, this gives us the result. 
\end{proof}

Cerf's proof of Smale's theorem (Corollary \ref{smthm}) is also highlighted in \cite{KK} \S 6.2.4.  When the co-dimension
of the embeddings are three or larger, sharp connectivity estimates for the scanning map exist.  See for example
\cite{BLR} pgs. 23--25, and the initial pages of Goodwillie's Ph.D thesis \cite{G}. The paper \cite{GKK} also includes
a detailed analysis of scanning maps for spaces of concordence embeddings, when the co-dimension is at least three.

The deloopings of the spaces $\Diff(D^n)$ and $\Emb(D^j, D^n)$ are studied in \cite{Sakai} and \cite{ST}. It would be interesting
to see if there are analogous retraction results for the deloopings of the scanning maps $\Diff(D^n) \to \Omega^{n-j} \Emb^{fr}(D^j, D^n)$. 
It is perhaps unlikely, but it is a basic question that deserves investigation.  

\begin{thm}\label{main3}The scanning map $\Emb(D^{n-1}, S^1 \times D^{n-1}) \to \Omega \Emb^\nu(D^{n-2}, S^1 \times D^{n-1})$ is the
inclusion portion of a homotopy-retraction, i.e. it induces split injections on all homotopy-groups for all $n \geq 2$. When
$n>2$ the $\nu$ can be dropped from the target space, i.e. the theorem remains true for embeddings without a normal vector field. 
\begin{proof}
By Theorem \ref{main1}, scanning gives us a homotopy equivalence 
$\Emb(D^{n-1}, S^1 \times D^{n-1}) \to \Omega \Emb^{\nu}(D^{n-2}, D^n)$. 
We construct an inclusion map $\Emb^{\nu}(D^{n-2}, D^n) \to \Emb^{\nu}(D^{n-2}, S^1 \times D^{n-1})$ by attaching a trivial
$1$-handle to $D^n$, i.e. thinking of $S^1 \times D^{n-1}$ as $D^n$ union a $1$-handle.  This inclusion is the inclusion portion
of a homotopy-retract, i.e. it has a left homotopy-inverse. The left homotopy inverse comes from lifting an embedding $D^{n-2} \to S^1 \times D^{n-1}$ 
to the universal cover, which we identify with a copy of $D^n$ with two points removed from the boundary.  
\end{proof}
\end{thm}

The proof of the above theorem is largely a duplicate of the proof of Theorem \ref{cerf1}. 
Similarly, this argument allows us to identify the map $\Emb(D^{n-1}, S^1 \times D^{n-1}) \to \Omega \Emb(D^{n-2}, S^1 \times D^{n-1})$
with the scanning map.  

Notice when $n=2$, the above scanning map is a homotopy-equivalence by Gramain \cite{Gramain}. When $n=3$ it follows 
by Hatcher's work \cite{Hat} that the scanning map is a homotopy-equivalence, indeed, both spaces are contractible.  

When $n \geq 4$ far less is known about such scanning maps. In \cite{BG} and \cite{BG2} the 
mapping-class group $\pi_0 \Diff(S^1 \times D^3)$ 
was shown to be not finitely generated via the map $\pi_0 \Diff(S^1 \times D^3) \to \pi_2 \Emb(D^1, S^1 \times D^3)$.  
Above, we see that the intermediate map
$\pi_0 \Diff(S^1 \times D^3) \to \pi_1 \Emb(D^2, S^1 \times D^3)$ has kernel isomorphic to $\pi_0 \Diff(D^4)$, this follows
from `handle-attachment homotopy-equivalence' $\Diff(S^1 \times D^{n-1}) \simeq \Diff(D^n) \times \Emb(D^{n-1}, S^1 \times D^{n-1})$ 
described in \cite{BG}.  The study of our scanning map $\pi_0 \Diff(S^1 \times D^3) \to \pi_2 \Emb(D^1, S^1 \times D^3)$ 
is thus reduced to the final step
$\pi_1 \Emb(D^2, S^1 \times D^3) \to \pi_2 \Emb(D^1, S^1 \times D^3)$, i.e. the loop space functor applied to the scanning map 

$$\Emb(D^2, S^1 \times D^3) \to \Omega \Emb(D^1, S^1 \times D^3).$$

One might attempt to apply the reductionist version of Theorem \ref{main1} to construct a homotopy-equivalence
$\Emb(D^2, S^1 \times D^3) \simeq \Omega \Emb^{\nu}(D^1, S^1 \times D^3 \cup H^3)$, but since the boundary circle for the embeddings
of $\Emb(D^2, S^1 \times D^3)$ is homologically trivial, it does not have the required 
$2$-sphere intersecting the embedding transversely in a single point.  
Alternatively, the embeddings in $\Emb(D^2, S^1 \times D^3)$ are not the cocore of a $2$-handle attachment, 
so we can not appeal to the primary version of Theorem \ref{main1}, either.  That said, we do know that the map 
$\Emb(D^2, S^1 \times D^3) \to \Omega \Emb(D^1, S^1 \times D^3)$
is homotopically non-trivial (\cite{BG}, \cite{BG2}) as the induced map on $\pi_1$ maps onto an infinitely-generated subgroup, 
so further study of these scanning maps is warranted.  

The work of Fresse-Turchin-Willwacher \cite{FTW} describes the delooping of the homotopy-fiber of 
the map from embeddings to immersions $\Emb(D^j, D^n) \to \Imm(D^j, D^n) \simeq \Omega^j V_{n,j}$, giving a fairly concrete
description of its  rational homotopy-type in the language of graph complexes when $n-j>2$.  In principle
this should give us some useful information on the co-dimension one scanning map 
$\Emb(D^j, D^n) \to \Omega \Emb(D^{j-1}, D^n)$ in rational homotopy, although our lack of understanding of
the induced map $\Emb(D^j, D^n) \to \Imm(D^j, D^n)$ in rational homotopy (when $j>1$) may be a limiting factor at present.   
A related topic is the `Freudenthal Suspension map' $\Emb(D^1, D^n) \to \Omega \Emb(D^1, D^{n+1})$ \cite{fam} 
which is defined via two canonical unknotting operations.  This map is known to be zero on rational homotopy 
(unpublished at this time), yet the map itself could potentially be homotopically non-trivial.


\section{Bott handles and miscellany}\label{miscsec}

The homotopy-equivalence $\Diff(D^n) \simeq \Omega \Emb(D^{n-1}, D^n)$ can be extrapolated to a homotopy-equivalence
$\Diff(I \times N) \simeq \Omega \Emb(\{\frac{1}{2}\} \times N, I \times N)$, and scanning maps

$$\Diff(D^k \times N) \to \Omega \Emb(D^{k-1} \times N,  D^k \times N) \to 
  \cdots \to \Omega^j \Emb(D^{k-j} \times N, D^k \times N).$$
 
Whereas the scanning of Section \ref{mainsec} could be viewed as an argument where the intermediate space is that of the space of
canceling handles, i.e. vanilla Morse theory, the scanning above has intermediate space the space Bott-style canceling handles, 
i.e. the kinds of
handles that occur with Bott-style Morse functions (functions on manifolds where the critical point sets are manifolds and
the Hessian is non-degenerate on these critical submanifolds \cite{Bott}).  For Bott-style Morse functions `handle' attachments 
are disc bundles over manifolds, whereas in standard Morse theory one attaches disc bundles 
over points, i.e. plain discs. Specifically, an adjunction where one attaches a disc-bundle over $M$, $M \ltimes D^k$ to another manifold 
$N$ along an embedding $M \ltimes \partial D^k \to \partial N$
is what is called a Bott-style handle attachment \cite{Bott}, as these sorts of attachments occur for Bott-type Morse functions, i.e. functions
$W \to \Real$ whose critical points are manifolds and the Hessian is non-degenerate on the normal bundle fibers. Bott-style
Morse functions typically occur when functions have symmetry, for example the trace of a matrix is a Bott-style Morse function
on the orthogonal group $O_n$. The critical points of this function are the square roots of the identity matrix $I$, thus copies of 
Grassman manifolds.  As a concrete example, the trace functional expresses $SO_3$ as the tautological line bundle over 
$\mathbb{R}P^2$ union a $3$-handle.
 
The analogue to Theorem \ref{main1} in the Bott case has the form of a homotopy-equivalence
$$\Emb(M \ltimes D^k, N \setminus \nu (M \ltimes D^{k-1})) \simeq \Omega \Emb(M \ltimes D^{k-1}, N).$$

Given that our scanning maps are highly structured, they would appear to be a potentially useful
device for exploring the homotopy-types of diffeomorphism groups like $\Diff(D^n)$, $\Diff(S^1 \times D^n)$ and generally
product manifolds $\Diff(N \times D^k)$, in particular for studying spaces of pseudoisotopies.  From this perspective 
there is perhaps a similarly overlooked element of Embedding Calculus \cite{MV} \cite{W} that is relevant. 

For example, given a manifold $M$, let $\mathcal{O}_k(M)$ be the category of open subsets of $M$ diffeomorphic to a disjoint union
of at most $k$ open balls, arrows given by inclusion maps.  Given $U \in \mathcal{O}_k(M)$ let 
$F(U)$ be $\Emb(U \times D^j, M \times D^j)$, i.e.
smooth embeddings of $U \times D^j$ in $M \times D^j$ that restrict to the standard inclusion on $U \times \partial D^j$.  
The $k$-th stage of the Taylor tower could be taken to be $T_k F(U) = \holim_{V \in \mathcal{O}_k(U)} F(V)$.   From this perspective, 
the scanning map is the evaluation map to the first stage of the Taylor tower.  Higher stages of the Taylor tower are built from 
spaces of generalized string links (in the sense that the Goodwillie-Weiss-Klein embedding calculus is built from 
configuration spaces),  and similarly the layers will be a relative section space.  This Taylor tower maps
to the GWK-Taylor tower so it should converge when the co-dimension of the embeddings are sufficiently large. Minimally
from the above it will have embeddings as a homotopy retract.   The
rate of convergence  of this Taylor tower we suspect will often be greater -- for example by Cerf's theorem
$\Diff(D^n) \simeq \Omega \Emb(D^{n-1}, D^n)$ the first stage when $M = I$ is homotopy-equivalent to 
$\Diff(I \times D^{j-1}) \simeq \Diff(D^j)$. 
The potential for this framework is that it may provide more manageable
inductive steps for practical computations of homotopy and homology groups of embedding spaces, 
as one is no longer comparing an embedding space 
directly with configuration spaces. Spaces of string links have been the subject of some recent investigations
by Koytcheff \cite{K1}, Turchin and Tsopm\'en\'e \cite{TPA} \cite{K1}, including a description of some of their low-dimensional 
homotopy groups \cite{K1} as well as an operad action \cite{K2}, so we may not be far removed from being able to analyze these 
{\it string link Taylor towers.}

String links appear in two essential ways in both \cite{BG} and 
\cite{BG2}.  Specifically, barbell diffeomorphism families are the induced diffeomorphisms coming from 
the low-dimensional homotopy groups of 
spaces of $2$-component string-links.  Moreover, the map we use to detect our diffeomorphisms of $S^1 \times D^{n-1}$ has the form 
$\Diff(S^1 \times D^{n-1}) \to \Omega^{n-2} \Emb(D^1, S^1 \times D^{n-1})$.  If we take the lifts of an element 
of $\Emb(D^1, S^1 \times D^{n-1})$ to the universal cover, we get an equivariant, infinite-component string link in 
$\Real \times D^{n-1}$.  Thus string links would appear to be a relatively efficient machine for investigating embedding
spaces and diffeomorphism groups.  It would be very interesting to see the relative rate of convergence of the above
Taylor towers, compared to the standard Embedding Calculus \cite{GKW}. 

A small comment on the relationship between the restriction maps $\Diff(S^n) \to \Emb(S^j, S^n)$ and the
Cerf half-disc fibrations.  When $j<n$ these fibrations are null-homotopic via a `shrinking support' argument
\cite{fam}.  This is closely related to the half-disc fibration.  Specifically, if we replace the above diffeomorphism
group and embedding space with their `long' version, and require the embeddings to have trivialized normal bundles
we have the fibration $\Diff(D^n) \to \Emb^{fr}(D^j, D^n)$.  This fibration has fiber homotopy-equivalent to 
$\Diff(S^{n-j-1} \times D^{j+1})$. There is a cancelling-handle homotopy-equivalence
$$\Diff(S^{n-j-1} \times D^{j+1}) \simeq \Diff(D^n) \times \Emb^{fr}(D^{j+1}, S^{n-j-1} \times D^{j+1}).$$
Lastly, let $\Emb^{fr}(\HD^{j+1}, D^n)$ be the half-disc embedding space where the half-discs are equipped with
trivialized normal bundles.  Then we have a fibre sequence
$$ \Emb^{fr}(D^{j+1}, S^{n-j-1} \times D^{j+1}) \to \Emb^{fr}(\HD^{j+1}, D^n) \to \Emb^{fr}(D^j, D^n). $$
Like in the unframed case, the space $\Emb^{fr}(\HD^{j+1}, D^n)$ is contractible.

This gives us a little commutative diagram of homotopy fiber sequences (three top vertical maps fibrations, 
three rightmost horizontal maps are also fibrations)
$$ \xymatrix{\Emb^{fr}(D^{j+1}, S^{n-j-1} \times D^{j+1}) \ar[r] & \Emb^{fr}(\HD^{j+1}, D^n) \ar[r] & \Emb^{fr}(D^j, D^n) \\
  \Diff(S^{n-j-1} \times D^{j+1}) \ar[r] \ar[u] & \Diff(D^n) \ar[r] \ar[u] & \Emb^{fr}(D^j, D^n) \ar[u] \\
  \Diff(D^n) \ar[r] \ar[u] & \Diff(D^n) \ar[r] \ar[u] & \{*\} \ar[u]}$$
i.e. we are asserting that the fibration $\Diff(D^n) \to \Emb^{fr}(D^j, D^n)$ is simply 
the half-disc fibration $\Emb^{fr}(\HD^{j+1}, D^n) \to \Emb^{fr}(D^j, D^n)$ but where we have inserted a trivial
$\Diff(D^n)$ factor in the total space and fiber. 

\section{The Sch\"onflies monoid}\label{schoen}

We end with the observation, implicit in \cite{BG}, that the monoid $\pi_0 \Emb(S^{n-1}, S^n)$ using the connect-sum operation, that
this is a group for all $n \geq 2$, as it is unclear if a proof of this statement exists in the literature. 
For $n \neq 4$ this group is known to be isomorphic to $\pi_0 \Diff(D^{n-1})$.  
In dimension $n=4$ the Sch\"onflies problem is equivalent to stating this group is trivial.  
  
The connect-sum operation on
$\pi_0 \Emb(S^{n-1}, S^n)$ has a description as a relative surgery (i.e. performing surgery on both the ambient manifold and submanifold
at the same time).  One embeds pairs $(D^n, D^{n-1})$ in the
pairs $(S^n, f(S^{n-1}))$ and $(S^n, g(S^{n-1}))$ respectively. 
Given that our embeddings are parametrized this requires a linearization
operation relative to the functions $f$ and $g$ about the embeddings $D^{n-1} \to f(S^{n-1})$ and
$D^{n-1} \to g(S^{n-1})$ respectively, as well as an identification of $S^n \# S^n$ with $S^n$.  

To minimize the overhead of formalism we will assume the homotopy-equivalence \cite{fam}
$$\Emb(S^{n-1}, S^n) \simeq SO_{n+1} \times \Emb(D^{n-1}, D^n)$$
which follows from a linearization argument.

This homotopy-equivalence tells us $\pi_0 \Emb(S^{n-1}, S^n) \simeq \pi_0 \Emb(D^{n-1}, D^n)$, 
allowing us to define the monoid structure on $\pi_0 \Emb(D^{n-1}, D^n)$.

The space $\Emb(D^{n-1}, D^n)$ can be thought of as the smooth embeddings $\Real^{n-1} \to \Real^n$ that agrees with the
standard inclusion $\{0\} \times \Real^{n-1} \subset \Real^n$ outside of $D^{n-1}$, and maps $D^{n-1}$ into $D^n$.  
We endow $\Emb(D^{n-1}, D^n)$ with a binary
operation (indeed many such) by stacking embeddings.  To stack two elements of $\Emb(D^{n-1}, D^n)$ one needs rescaling and 
translation to make the operation precise \cite{fam}. 

If one combines all such operations on has an action of the operad of $(n-1)$-discs on 
$\Emb(D^{n-1}, D^n)$.  The connect-sum operation 
is the induced monoidal structure on $\pi_0 \Emb(D^{n-1}, D^n)$. The neutral element is the linear embedding.  
This connect-sum operation generalizes 
directly to all embedding spaces $\pi_0 \Emb(S^j, S^n)$. When $n=j+2$ it is the classical connect-sum 
operation on co-dimension two knots, and when $j=n$ it is the composition operation on $\pi_0 \Diff(S^n)$. 
    
The proof is a small variation on
the proofs of Proposition \ref{cerf1} and Theorem \ref{main3}. 
 The inclusion from Proposition \ref{cerf1} 

$$\Emb(D^{n-1}, D^n) \to \Emb(D^{n-1}, S^1 \times D^{n-1})$$ 

is compatible with stacking, i.e. it induces a map
of monoids on path-components.  The space $\Emb(D^{n-1}, S^1 \times D^{n-1})$ has the homotopy-type of $\Omega \Emb^{\nu}(D^{n-2}, D^n)$
by Theorem \ref{main1}.  The space $\Omega \Emb^{\nu}(D^{n-2}, D^n)$ has {\it two} stacking operations, i.e. one can `stack' 
using the loop-space parameter,  or stack using the analogous stacking operation on the space $\Emb^{\nu}(D^{n-2}, D^n)$.  These 
two operations are homotopic.  In introductory algebraic topology courses, 
one uses this type of argument to show the fundamental group of a topological group must be abelian. It is often called an Eckmann-Hilton
argument. Another way to say this is that the space $\Omega \Emb^{\nu}(D^{n-2}, D^n)$ has an action of the operad of $2$-cubes, where
the action restricts to either concatenation construction, depending on the position of the cubes.

\begin{thm}\label{monoidstruc} The monoid structure on $\pi_0 \Emb(S^{n-1}, S^n)$ coming from the connect-sum operation, this is
a group for all $n \geq 2$.  Moreover, there is an onto-homomorphism 
$$\pi_1 \Emb^{\nu}(D^{n-2}, D^n) \to \pi_0 \Emb(D^{n-1}, D^n) \simeq \pi_0 \Emb(S^{n-1}, S^n).$$
\end{thm}

When $n=1$, the set $\pi_0 \Emb(S^0, S^1)$ is also known to be a group, as it has only a single-element. The group
$\pi_1 \Emb^{\nu}(D^{n-2}, D^n)$ is known to be non-trivial when $n=4$ \cite{BG} although all presently-known elements map
to zero in $\pi_0 \Emb(D^{n-1}, D^n)$.  

The homomorphism $\pi_1 \Emb^{\nu}(D^{n-2}, D^n) \to \pi_0 \Emb(D^{n-1}, D^n)$ has this description.  Take a linearly-embedded
copy of $\HD^{n-1}$ in $D^n$, i.e. the half-disc in $D^{n-1} \times \{0\} \subset D^n$. Given a loop of embeddings of
$D^{n-2}$ (with normal vector field) in $D^n$, lift that path of embeddings to a path in $\Emb(\HD^{n-1}, D^n)$ that begins at 
the linear embedding.  At the end of this path, we have 
a smooth embedding $\HD^{n-1} \to D^n$ which agrees with our standard
inclusion on the boundary, including its normal derivative.  Via a small isotopy, we can ensure this embedding $\HD^{n-1} \to D^n$
 agrees with the standard inclusion in a neighborhood of
the boundary.  Drill the flat face of the embedded $\HD^{n-1}$ from $D^n$, this results in a copy of $S^1 \times D^{n-1}$
together with a smoothly-embedded $D^{n-1} \to S^1 \times D^{n-1}$ which agrees with the standard inclusion 
$\{1\} \times D^{n-1} \subset S^1 \times D^{n-1}$
on the boundary. Lift this embedding to the universal cover of $S^1 \times D^{n-1}$ and identify the universal cover 
with a subspace of $D^n$ ($D^n$ with two boundary points removed). This embedding $D^{n-1} \to D^n$ is the value of our map
$\pi_1 \Emb^{\nu}(D^{n-2}, D^n) \to \pi_0 \Emb(D^{n-1}, D^n)$. 

There is a Kervaire-Milnor style argument that the monoid $\pi_0 \Emb(S^{n-1}, S^n)$ has inverses.  Given an
embedding $f : S^{n-1} \to S^n$ drill a small open ball from $S^{n-1}$ and consider a tubular neighborhood of this manifold. 
It is diffeomorphic to $D^{n-1} \times I$, and so the boundary of this manifold is diffeomorphic to the connect-sum of $f(S^{n-1})$ 
with its mirror-reverse. Since the embedding bounds a copy of $D^{n-1} \times I \simeq D^n$ (after rounding corners), we have
that $f(S^{n-1}) \# f(-S^{n-1})$ is standard, thus $f$ and its mirror-reverse are inverses of each other.  The relative advantage
of Theorem \ref{monoidstruc} is that it provides a group $\pi_1 \Emb^\nu(D^{n-2}, D^n)$ that maps onto the
Sch\"onflies monoid $\pi_0 \Emb(S^{n-1}, S^n)$, i.e. it gives us a prescription for how one can construct all Sch\"onflies 
spheres. 

The resolution of the Sch\"onflies problem in dimension different from four gives another argument that the monoid of
Sch\"onflies spheres $\pi_0 \Emb(S^{n-1}, S^n)$ is a group, when $n \neq 4$, as this tells us the reparametrizations of the linear
embedding gives an onto homomorphism $\pi_0 \Diff(S^{n-1}) \to \pi_0 \Emb(S^{n-1}, S^n)$.  
The triviality of $\pi_0 \Emb(S^{n-1}, S^n)$ when $n=4$ is equivalent to the Sch\"onflies Problem, as $\Diff(D^3)$ is 
contractible \cite{Hat}. 

\providecommand{\bysame}{\leavevmode\hbox to3em{\hrulefill}\thinspace}

\end{document}